\pdfoutput=1
\documentclass[11pt, a4paper]{amsart}
\usepackage[T1]{fontenc}
\usepackage[utf8]{inputenc} 
\usepackage{amssymb, amsmath, amsthm, bm, hyperref}
\usepackage[all, cmtip]{xy}
\usepackage{microtype}

\newtheorem*{MTh}{Main Theorem}

\newtheorem{Lem}{Lemma}
\newtheorem{Prop}{Proposition}

\newtheorem*{Cor*}{Corollary}

\theoremstyle{definition}

\newtheorem{Def}{Definition}

\theoremstyle{remark}
\newtheorem{Rm}{Remark}
\newtheorem{Ex}{Example}

\newcommand{\Z}{\mathbb{Z}}

\newcommand{\ring}{\Bbbk} 
\mathchardef\mhyphen="2D 

\DeclareMathOperator{\Nat}{Nat} 
\DeclareMathOperator{\Ob}{Ob} 
\DeclareMathOperator{\Map}{Map} 
\DeclareMathOperator{\hc}{hc} 
\newcommand{\op}{\mathrm{op}} 
\DeclareMathOperator{\cyc}{cyc}

\newcommand{\Category}[1]{\textsf{\textbf{#1}}} 
\newcommand{\Ch}{\Category{Ch}} 
\newcommand{\ford}{\Delta} 
\newcommand{\cyclic}{\Lambda} 

\newcommand{\ChHom}{\underline{\Ch}}
\newcommand{\SeqOp}{\mathcal{S}} 
\newcommand{\loops}{L}

\begin{document}
\title{Free loop space and the cyclic bar construction}
\author{Massimiliano Ungheretti}
\address{\newline Department of Mathematical Sciences\newline University of Copenhagen\newline
Universitetsparken 5\newline 2100 Copenhagen, Denmark}
\email{m.ungheretti@math.ku.dk}
\date{Uploaded February 29, 2016. Revised August 2016}
\urladdr{\url{www.ungheretti.com}}



\begin{abstract}Using the $E_\infty\mhyphen$structure on singular cochains, we construct a homotopy coherent map from the cyclic bar construction of the differential graded algebra of cochains on a space to a model for the cochains on its free loop space. This fills a gap in the paper ``Cyclic homology and equivariant homology'' by John~D.S.~Jones.\end{abstract}

\maketitle

\section{Introduction}
Hochschild homology has been widely used to provide an algebraic model for the cohomology of free loop spaces. In particular, there is an isomorphism $HH_*(S^*(X))\cong H^*(\loops X)$ for the singular cochain algebra $S^*(X)$ of a simply connected space $X$. This was proved by John~D.S.~Jones in \cite{Jones}, together with its $SO(2)\mhyphen$equivariant version.

One step in the proof of this isomorphism requires one to establish the equivalence of two diagrams of chain complexes, $B^{\cyc}_\bullet S^*(X)$ and $S^*(\Map (S^1_{\bullet}, X))$. The first is the cyclic bar construction of cochains on a space $X$ and the second is given by the cochains on $\Map (S^1_{\bullet}, X)$, a cocyclic space modelling the free loop space of $X$. Jones uses the Alexander--Whitney map to compare these two cyclic objects, which gives a map on every simplicial level; however, it does not form a map of cyclic objects, as it does not commute with the structure maps of the cyclic category. This fact can already be seen in simplicial level one, where the Alexander--Whitney map should be symmetric on the cochain level in order to commute with the cyclic operator $t$ and the first boundary map $d_1$. This cannot be the case, as it would imply that the cup product is commutative on the cochain level. The cup product is, however, commutative up to coherent homotopy and it is this natural $E_\infty\mhyphen$structure that will be used to construct a homotopy coherent isomorphism instead, filling the gap in the proof.

\begin{MTh}Let $X$ be a space with finite type homology over a principal ideal domain $\ring$. There is a natural zigzag of equivalences of cyclic chain complexes \[B^{\cyc}_\bullet S^*(X;\ring) \xleftarrow{\simeq} QB^{\cyc}_\bullet S^*(X;\ring) \xrightarrow{\simeq} S^*(\Map (S^1_{\bullet}, X);\ring),\]
where $QB^{\cyc}_\bullet S^*(X;\ring)$ is a resolution of the cyclic bar construction.\end{MTh}
\begin{Rm}The finiteness assumption is not needed when working with chains rather than cochains: the cyclic cobar construction of the coalgebra of chains $\Omega^{\cyc}_\bullet S_*(X)$ is equivalent to $S_*(\Map (S^1_{\bullet}, X))$. This statement works over the integers and uses the same proof combined with the observation that the $E_\infty\mhyphen$structure on cochains described in \cite{McClureSmithCochains} is the linear dual of an operad coaction on chains.\end{Rm}

If $X$ is simply connected and of finite type over a field $\ring$, Jones' proof implies the isomorphisms
\begin{align*}
H^*(\loops X;\ring) &\cong HH_*(S^*(X;\ring)) \\
H^*(\loops X \times_{\mathrm{SO}(2)} E\mathrm{SO} (2);\ring) &\cong HC_*^-(S^*(X;\ring)).
\end{align*}

The assumption that $X$ is of finite type over a field $\ring$ is not explicitly stated in \cite{Jones}, but it is used in a cited paper: That $\ring$ is a field is assumed in \cite{AndersonEMSS} to establish the ``convergence'' of the cosimplicial mapping space $\Map (S^1_{\bullet}, X)$ over $\ring$. The finite type assumption ensures that the Alexander--Whitney map $S^*(X)\otimes S^*(X) \rightarrow S^*(X\times X)$ is a quasi isomorphism.

In Proposition 5.3 of \cite{AyalaFrancis}, the authors use factorization homology over $S^1$ to show that the assumptions can be weakened to allow $\ring$ to be any commutative ring, $X$ a nilpotent space equivalent to a finite type CW complex and $\pi_1$ finite.

From the algebraic theorem in \cite{JonesMcCleary}, one can reprove both isomorphisms using the same conditions. For this one needs to start with the isomorphism $H_*(\loops X)\cong HH_*(S_*(\Omega X))$ from \cite{Goodwillie} and combine this with Adams' cobar equivalence.

A failure to commute with the last boundary map also appears in the papers \cite{PatchkoriaSagave} and \cite{WahlILSMCG}, where methods similar to ours are used.

\subsection{Acknowledgements}
The author would like to thank John Jones for the useful correspondence, Kristian Moi for discussing Section \ref{sec:hc} and Nathalie Wahl for general guidance. The author was supported by the Danish National Sciences Research Council (DNSRC) and the European Research Council (ERC), as well as by the Danish National Research Foundation (DNRF) through the Centre for Symmetry and Deformation.

\subsection{Conventions}
We use the closed symmetric monoidal structure of $\Ch$, the category of homologically graded chain complexes of abelian groups. The tensor product of two chain complexes carries a differential $d(x\otimes y)=dx\otimes y + (-1)^{|x|} x\otimes dy$. The internal hom is a chain complex $\ChHom(X,Y)$ that in degree $n$ consists of linear maps of degree $n$ and has differential $(d\psi) (x) = d\circ\psi(x) - (-1)^n \psi(dx)$. This means that the chain maps are the 0-cycles in this chain complex. Any cosimplicial object in the category of chain complexes $A^\bullet$ gives a double complex using $\Sigma (-1)^i \delta^i$. Its (product) totalization is written as $A$, omitting the bullet. We use similar notation for simplicial chain complexes $A_\bullet$.

\subsection{Cyclic objects} We briefly recall some definitions of cyclic objects and refer to \cite{Jones,LodayCyclic} for more details.
The morphisms of the category of finite ordered sets $\ford$ are generated by $\delta^i, \sigma^i$, which satisfy the simplicial relations. By appropriately adding cyclic permutations $\langle \tau \rangle = C_{n+1}$ as the automorphisms of $[n]$, one obtains Connes' cyclic category $\cyclic$. Functors out of this category are called (co)cyclic objects.
\begin{Ex}There is a cyclic set $[n]\mapsto S^1_n = \Z/(n+1)\Z$ that realizes to the circle. From this, one obtains for each space $X$ a cocyclic space $[n] \mapsto \Map(S^1_n, X)= X^{n+1}$ which totalizes to the free loop space $\loops X$. The coboundaries are given by the diagonal maps, the codegeneracies by forgetting factors and the cyclic maps by cyclically permuting the factors. For example, \[\delta^{n+1}(x_0,\ldots,x_n)=(x_0,x_1,\ldots,x_n,x_0).\] By functoriality of $S^*(\mhyphen)$, $S^*(\Map(S^1_\bullet,X))$ is a cyclic chain complex.\end{Ex}
\begin{Ex}For any unital differential graded algebra $A$, we have the \emph{cyclic bar complex} $(B^{\cyc}A)[n]=A^{\otimes n+1}$, that can be used to compute Hochschild and cyclic homology. The structure maps are given by multiplication, insertions of the unit and cyclic permutations of the tensor factors.\end{Ex}

\subsection{Homotopy commutative structure of cochains}\label{sec:cochains}
The main ingredient for the proof of our main theorem is the natural $E_\infty\mhyphen$structure on the normalized singular cochains $S^*(X)$. Such operad actions are given, for example, in \cite{BergerFresseCochains,McClureSmithCochains} and are the integral analogue of Sullivan and Quillen models \cite{MandellCochains}. In our proofs, we only use the fact that there exists a symmetric differential graded operad $\SeqOp$ which has the homology of a point in every arity and which comes with a map from the unital associative operad and a map to the natural operations on $S^*(X)$ which specify the cup product and its unit.

\begin{Rm}An inductive argument for the contractibility of an operad $\SeqOp$ is given on p. 689 of \cite{McClureSmithCochains}. However, there is a minor mistake that may be spotted by applying the formula $\partial s + s \partial = \textit{id} + \iota r$ to the example $\langle  3123 \rangle$. To fix this, it is enough to change the map $r$ to only be 0 unless the sequence contains exactly a single 1.\end{Rm}

\section{Homotopy coherent natural transformations}\label{sec:hc}
In this section we adapt the treatment of homotopy coherent natural transformation in \cite[\S8]{DuggerHocolim} from spaces to chain complexes.
\begin{Def}Let $I$ be a small category and $F,G\colon I\rightarrow \Ch$ two diagrams of chain complexes. Define the cosimplicial chain complex of \emph{homotopy coherent natural transformations} $\hc(F,G)^\bullet\colon \ford \rightarrow \Ch$ as \[\hc(F,G)^n=\prod_{\underline{\phi}\in N_nI} \ChHom (F(i_0),G(i_n)),\] where the product runs over simplices of the nerve $N_\bullet I$ of $I$. The structure maps on such families $A\in \hc(F,G)^n$ are given by
\[(\sigma^i A)_{\underline{\phi}} = A_{s_i \underline{\phi}}\]
\[ 
 (\delta^i A)_{\underline{\phi}} = 
  \begin{cases} 
   F(i_0) \xrightarrow{F(i_0\rightarrow i_1)} F(i_1) \xrightarrow{ A_{d_0 {\underline{\phi}}}} G(i_{n+1}) & \text{if } i=0, \\
   F(i_0) \xrightarrow{A_{d_i {\underline{\phi}}}} G(i_{n+1}) & \text{if } 0<i<n+1, \\
   F(i_0) \xrightarrow{A_{d_{n+1} \underline{\phi}} } G(i_n) \xrightarrow{G(i_{n}\rightarrow i_{n+1})} G(i_{n+1}) & \text{if } i=n+1.
  \end{cases}
\]
A single \emph{homotopy coherent natural transformation} is defined to be a $0 \mhyphen$cycle in the totalization $\hc(F,G)$.
\end{Def}

\begin{Ex}Finding single homotopy coherent natural transformation means finding a family of elements $A^n \in \hc(F,G)^n$ of degree $n$ such that $d A^0=0$ and $\Sigma_i (-1)^{i} \delta^i A^n = (-1)^n d A^{n+1}$. These $A^n$ are themselves families indexed by $\underline{\phi}\in N_nI$, which we write as $A_{\underline{\phi}}\in \ChHom(F(i_0),G(i_n))$. For $n=0$ this means that we have $A_i \colon F(i)\rightarrow G(i)$ a chain map of degree zero for each object $i\in I$. In the case when $n=1$, we have for each morphism $\phi\colon i_0\rightarrow i_1$ in $I$ a map $A_\phi \colon F(i_0)\rightarrow G(i_1)$ of degree one. These maps are not required to be chain maps but instead satisfy \[ d_{G(i_1)}\circ A_\phi + A_\phi \circ d_{F(i_0)}=(\delta^0 A)_\phi - (\delta^1 A)_\phi = A_{i_1} \circ F(\phi) - G(\phi)\circ A_{i_0}.\] That is, the maps $A_\phi$ provide chain homotopies implementing the failure of the naturality squares to commute on the nose.
\[ \xymatrix{ F(i_0) \ar[r]^{F(\phi)} \ar[d]_{A_{i_0}} & F(i_1) \ar[d]^{A_{i_1}}\\ G(i_0) \ar[r]^{G(\phi)} & G(i_1) }\]
For $n\geq 1$, we have homotopies relating all the different ways of composing a string of morphisms and the previous homotopies.
\end{Ex}

\begin{Def} For a small diagram $F\colon I\rightarrow \Ch$, we define the resolution $QF_\bullet\colon I\times \ford^{\op}\rightarrow \Ch$ as the two-sided bar construction $QF_\bullet = B_\bullet (I,I,F)$, where the first $I$ is shorthand for the bifunctor $\Z I(\mhyphen,\mhyphen)$. Concretely, this gives a simplicial $I\mhyphen$diagram, which at the object $i\in \Ob I$ in simplicial degree $n$ is a sum over $n\mhyphen$simplices in $N_\bullet(I/i)$.
\[QF_n(i)=\bigoplus_{i_0\rightarrow \ldots \rightarrow i_n \rightarrow i} F(i_0)\]
For a morphism $\alpha\colon i\rightarrow j$, we have a map $QF_\bullet (i)\rightarrow QF_\bullet(j)$ induced by $\alpha_*\colon N_\bullet(I/i)\rightarrow N_\bullet(I/j)$. The simplicial structure maps all act on the indexing sets $N_\bullet(I/i)$, where a composition with $F(i_0\rightarrow i_1)$ is needed in the definition of $d_0$.
\end{Def}

\begin{Prop}\label{prop:Q}For a small diagram $F\colon I\rightarrow \Ch$, the resolution $QF_\bullet \colon I\times \ford^{\op}\rightarrow \Ch$ has the following properties:
\begin{enumerate}
\item\label{Prop:Q_hc} There is a canonical isomorphism of cosimplicial chain complexes 
\[ \alpha^\bullet\colon \underline{\Nat}_I(QF_\bullet,G) \xrightarrow{\cong} \hc(F,G)^\bullet.\]
\item\label{Prop:Q_augm} There is a natural object-wise quasi isomorphism $QF\xrightarrow{\simeq} F$.
\item\label{Prop:Q_qi} Under the identification of total complexes 
\[\alpha\colon \underline{\Nat}_I(QF,G) \xrightarrow{\cong} \hc(F,G),\]
the quasi isomorphisms on the left hand side correspond on the right hand side to the $0\mhyphen$cycles that in cosimplicial degree 0 give quasi isomorphisms $F(i)\xrightarrow{\simeq}G(i)$.
\end{enumerate}\end{Prop}
\begin{proof}\leavevmode
\begin{enumerate}
\item On cosimplicial level $n$, the left hand side is a subset
\[\underline{\Nat}_I(QF_n,G)\subset \prod_{i\in \Ob I} \ChHom( \hspace{-1em} \bigoplus_{i_0\rightarrow \ldots \rightarrow i_n\rightarrow i}\hspace{-1em} F(i_0),G(i))= \hspace{-1em} \prod_{i_0\rightarrow \ldots \rightarrow i_n \rightarrow i}\hspace{-1em} \ChHom (F(i_0),G(i)).\]
It is exactly the subset determined by a naturality condition, comparing $\alpha_*\colon N_\bullet (I/i)\rightarrow N_\bullet (I/j)$ with $G(\alpha)$ for morphisms $\alpha\colon i\rightarrow j$. One sees that this amounts exactly to the data being determined by the simplices of the form $i_0\rightarrow \ldots \rightarrow i_n \xrightarrow{id} i_n$. To obtain the value at $i_0\rightarrow \ldots \rightarrow i_n \rightarrow i$, post-compose by $G(i_n\rightarrow i)$.

It remains to compare the cosimplicial structure maps. The $\delta^j$ for $j\neq 0,n$ are clear as they only affect the index. Since $d_0$ used $F(i_0\rightarrow i_1)$, we see that $\delta^0$ precomposes by this map. For the last coboundary, we need to use the naturality to see the postcomposition by $G(i_{n-1}\rightarrow i_n)$.
\item There is a standard augmentation $QF\rightarrow F$ defined as
\[\sum F(i_0\rightarrow i) \colon \bigoplus_{i_0\rightarrow \ldots \rightarrow i_n \rightarrow i} F(i_0)\rightarrow F(i),\]
with contracting homotopy $s_{n+1}$ given by appending the identity at the end of the indexing simplex.
\item First, observe that the chain maps are the $0\mhyphen$cycles. The behaviour of a map $QF(i)\rightarrow G(i)$ in homology is determined by what it does in cosimplicial degree 0. This can be seen by inspecting the augmentation and $H_*(F)\cong H_*(QF)\rightarrow H_*(G)$.
\end{enumerate}
\end{proof}

\section{Comparing the cyclic chain complexes}
In this section we construct a homotopy coherent natural transformation for $I=\cyclic^{\op}, F=B^{\cyc}_\bullet S^*(X),G = S^*(\Map (S^1_{\bullet}, X))\colon \cyclic^{\op}\rightarrow \Ch$.
\begin{Prop}\label{prop:hc_exists}There exists a homotopy coherent natural transformation $A$ from $B^{\cyc}_\bullet S^*(X)$ to $S^*(\Map (S^1_{\bullet}, X))$ that at each object $i$ is given by the Alexander--Whitney maps
\[B^{\cyc} S^*(X)(i)= S^*(X)^{\otimes i+1}\rightarrow S^*(X^{i+1})=S^*(\Map (S^1_{\bullet}, X))(i)\]
and is natural in $X$.\end{Prop}
Before giving the proof of this proposition, we introduce some notation and prove a lemma. For any $\underline{\phi}=(i_0 \xrightarrow{\phi_1} i_1 \ldots \xrightarrow{\phi_m} i_m)\in N_m\cyclic^{\op}$, we write $\phi$ for the composition $\phi_m\circ \ldots \circ \phi_1$. Associated to $\phi\in \cyclic^{\op}([i],[j])$ are three structure maps of (co)cyclic objects:
\begin{align*}
&\phi_* \colon S^*(X)^{\otimes i+1}\rightarrow S^*(X)^{\otimes j+1}\\
&\phi \colon X^{j+1} \rightarrow X^{i+1}\\
&\phi^* \colon S^*(X^{i+1})\rightarrow S^*(X^{j+1}).
\end{align*}
Using the projections $\pi_k\colon X^{i+1}\rightarrow X, k=0,\ldots, i$, we moreover associate to each $\phi\in \cyclic^{\op}([i],[j])$ a map 
\[\pi_\phi=(\pi_0\circ \phi)^*\otimes \ldots \otimes (\pi_i \circ \phi)^*\colon S^*(X)^{\otimes i+1}\rightarrow S^*(X^{j+1})^{\otimes i+1}.\]

\begin{Lem}\label{lem:cyclic_interaction} The following square commutes for any $\phi \in \cyclic^{\op}([i],[j])$.
\[\xymatrix{
	S^*(X)^{\otimes i+1} \ar[r]^{\phi_*} \ar[d]_{\pi_\phi} & S^*(X)^{\otimes j+1} \ar[d]_{\pi_{id}} \\
	S^*(X^{j+1})^{\otimes i+1} \ar[r]^{\phi_*} & S^*(X^{j+1})^{\otimes j+1}
}\]
\end{Lem}
\begin{proof}
This is an elementary check for the boundaries, degeneracies and cyclic operators, which together generate all morphisms in $\cyclic^{\op}$. If we have two composable morphisms $\psi\in \cyclic^{\op}([i],[j]), \phi\in \cyclic^{\op}([j],[k])$ that both satisfy the condition, then their composition satisfies the condition.

\[\xymatrix{
	S^*(X)^{\otimes i+1} \ar[d]_{\pi_{\psi}} \ar[r]^{\psi_*} & S^*(X)^{\otimes j+1} \ar[d]_{\pi_{id}} \ar[r]^{\phi_*} & S^*(X)^{\otimes k+1} \ar[dd]_{\pi_{id}}\\
	S^*(X^{j+1})^{\otimes i+1} \ar[d]_{(\phi^*)^{\otimes i+1}} \ar[r]^{\psi_*} & S^*(X^{j+1})^{\otimes j+1} \ar[d]_{(\phi^*)^{\otimes j+1}} & \\
	S^*(X^{k+1})^{\otimes i+1} \ar[r]^{\psi_*} & S^*(X^{k+1})^{\otimes j+1} \ar[r]^{\phi_*} & S^*(X^{k+1})^{\otimes k+1}
}\]
Note that $\pi_\phi= (\phi^*)^{\otimes i+1} \circ \pi_{id}$ and $\pi_{\psi \circ \phi}= (\phi^*)^{\otimes j+1} \circ \pi_\psi$. The top left and right hand squares commute by assumption on $\psi$ and $\phi$ respectively. The bottom left square commutes by naturality of the cyclic bar construction with respect to the map $\phi^*$ of differential graded algebras.
\end{proof}

\begin{proof}[Proof of Proposition~\ref{prop:hc_exists}] Fix a contractible operad $\SeqOp$ acting naturally on $S^*(X)$ as described in Section~\ref{sec:cochains}. We will prove the existence of the homotopy coherent map $A$ by induction on the cosimplicial degree $m$. For $m=0$ we have the Alexander--Whitney maps, which are given as \[ A_i^0\colon S^*(X)^{\otimes i+1} \xrightarrow{\pi_{id}} S^*(X^{i+1})^{\otimes i+1} \xrightarrow{S_i} S^*(X^{i+1}).\]
Here $S_i$ is the $i+1$ fold cup product, which lives in the operad as $S_i\in\SeqOp (i+1)_0$. The fact that this map can be factored as such will be the essential idea of the proof. Assume that we have defined $A^m$ for $m < n$ satisfying the boundary condition $\sum (-1)^j (\delta^j A^{m-1})_{\underline{\phi}}=(-1)^m d A^m_{\underline{\phi}}$. Assume furthermore that the components are of the form $A_{\underline{\phi}}=S_{\underline{\phi}} \circ \pi_{\phi}$ for $\underline{\phi}\in N_m \cyclic^{\op}$ with composition $\phi$ and the $S_{\underline{\phi}} \in \SeqOp(i_0+1)_m$ satisfying 
\begin{equation}\label{eq:S_condition} \tag{$\star$} (-1)^m dS_{\underline{\phi}}= S_{d_0 \underline{\phi}} \circ \phi_{1*} + \sum_{j=1}^m (-1)^j S_{d_j \underline{\phi}}. \end{equation}

To show that we can extend this construction to level $n$, we need to find $A_{\underline{\phi}}$ of this form for all $\underline{\phi} \in N_n \cyclic^{\op}$ in such a way that the boundary condition holds and $S_{\underline{\phi}}$ satisfies (\ref{eq:S_condition}). To do this, we describe the cosimplicial differential of $\hc(B^{\cyc}S^*(X),S^*(\Map (S^1_{\bullet}, X)))$ to see that (\ref{eq:S_condition}) implies the boundary condition.

The first coboundary can be written as $(\delta^0A)_{\underline{\phi}}=A_{d_0 \underline{\phi}}\circ \phi_{1*}=S_{d_0\underline{\phi}}\circ \pi_{d_0\phi} \circ \phi_{1*}= S_{d_0 \underline{\phi}} \circ \phi_{1*} \circ \pi_\phi$. The last equality can be seen using the commuting diagram
\[\xymatrixcolsep{4pc}\xymatrix{
	S^*(X)^{\otimes{i_0+1}} \ar[r]^{\pi_{\phi_1}} \ar[d]_{\phi_{1 *}} & S^*(X^{i_1+1})^{\otimes{i_0+1}} \ar[r]^{(\hat{\phi}^*)^{\otimes i_0+1}} \ar[d]_{\phi_{1*}} & S^*(X^{i_m+1})^{\otimes{i_0+1}} \ar[d]_{\phi_{1*}}\\
	S^*(X)^{\otimes{i_1+1}} \ar[r]^{\pi_{id}} & S^*(X^{i_1+1})^{\otimes{i_1+1}} \ar[r]^{(\hat{\phi}^*)^{\otimes i_1+1}} & S^*(X^{i_m+1})^{\otimes{i_m+1}},
}\]
where $\hat{\phi}=\phi_m\circ \ldots \circ \phi_2$ is associated to $d_0\underline{\phi}$. The first square commutes by Lemma~\ref{lem:cyclic_interaction} and the second by naturality of the cyclic bar construction.

The last coboundary can be factored as 
\[(\delta^m A)_{ \underline{\phi} } = {\phi_m}^* \circ A_{{d_m} {\underline{\phi}}} = {{\phi}_m}^* \circ S_{d_m {\underline{\phi}}} \circ \pi_{\tilde{\phi}} = S_{d_m {\underline{\phi}}} \circ {\pi}_{\phi}.\] 
Associated to the $d_m {\underline{\phi}}$ is the composition $\tilde{\phi} = \phi_{m-1} \circ \ldots \circ \phi_1$ and the last equality is a consequence of the commutativity of the diagram.
\[\xymatrixcolsep{5pc}\xymatrix{
	S^*(X)^{\otimes{i_0+1}} \ar[r]^{\pi_{\tilde{\phi}}} \ar[rd]_{\pi_\phi} & S^*(X^{i_{m-1}+1})^{\otimes{i_0+1}} \ar[r]^{S_{d_m \underline{\phi}}} \ar[d]^{({\phi_m}^*)^{\otimes i_0+1}} & S^*(X^{i_{m-1}+1}) \ar[d]^{{\phi_m}^*}\\
							& S^*(X^{i_{m}+1})^{\otimes{i_0+1}} \ar[r]^{S_{d_m \underline{\phi}}} & S^*(X^{i_{m}+1})}\]
The square commutes by the naturality of the operation $S_{d_m \underline{\phi}}$.

All the intermediate coboundaries $(\delta^j A)_{\underline{\phi}}$ for $j\neq 0,m$ are already of the form $(\delta^jA)_{\underline{\phi}}= A_{d_j\underline{\phi}} = S_{d_j \underline{\phi}} \circ \pi_{\phi}$. Observe that $\pi_{d_j\underline{\phi}}=\pi_\phi$.

This shows that (\ref{eq:S_condition}) implies the boundary condition. Also, the expression (\ref{eq:S_condition}) lives entirely inside $\SeqOp$ since $\phi_*$ is a composition of cup products, insertions of identities and permutations of arguments. One can see that the right hand side is in fact a cycle in $\SeqOp$ by applying the differential termwise and using the inductive hypotheses. This produces terms like $S_{d_0 d_0 \underline{\phi}}\circ \phi_{2 *} \circ \phi_{1 *}$ and $S_{d_l d_j \underline{\phi}}$ which all cancel out by the simplicial identities. As $\SeqOp$ has the homology of a point in every arity, this implies that such $S_{\underline{\phi}}$ exist.
\end{proof}

\begin{proof}[Proof of Main Theorem] The augmentation of Proposition \ref{prop:Q}.\ref{Prop:Q_augm} provides the first quasi isomorphism $QB^{\cyc}_\bullet S^*(X;\ring) \xrightarrow{\simeq} B^{\cyc}_\bullet S^*(X;\ring)$. The second map $QB^{\cyc}_\bullet S^*(X;\ring) \rightarrow S^*(\Map (S^1_{\bullet}, X);\ring)$ is provided by Propositions~\ref{prop:Q}.\ref{Prop:Q_hc} and \ref{prop:hc_exists}. The finiteness assumptions imply that the Alexander--Whitney maps are quasi isomorphisms, meaning we can apply Propositions~\ref{prop:Q}.\ref{Prop:Q_qi} to see that the map is a quasi isomorphism.\end{proof}
\bibliographystyle{alpha}
\bibliography{hochschild}

\end{document}